\documentclass[12pt]{article}
\usepackage{amsmath}
\usepackage{amssymb}
\usepackage{amsthm}

\title{Generators of the homological Goldman Lie algebra}
\author{Nariya Kawazumi\thanks{partially supported by the Grant-in-Aid for
Scientific Research (A) (No.20244003), (A) (No.22244005)
and (B) (No.24340010) from the
Japan Society for Promotion of Sciences}, 
Yusuke Kuno,
Kazuki Toda
}
\date{}

\theoremstyle{definition}
\newtheorem{Def}{Definition}[section]
\newtheorem{rem}[Def]{Remark}

\theoremstyle{plain}
\newtheorem{Prop}[Def]{Propositon}
\newtheorem{Lem}[Def]{Lemma}
\newtheorem{Thm}[Def]{Theorem}

\def\mbN{\mathbb{N}}
\def\mbZ{\mathbb{Z}}
\def\mbQ{\mathbb{Q}}

\begin{document}
\maketitle

\begin{abstract}
We determine the minimal number of generators of
the homological Goldman Lie algebra of a surface
consisting of elements of the first homology group of the surface.
\end{abstract} 

\section{Introduction}
Let $\Sigma$ be a compact connected oriented surface of genus $g\geq 1$.
The first homology group $H=H_1(\Sigma ,\mbZ )$ of $\Sigma$
is equipped with a skew-symmetric bilinear form
$\langle -,-\rangle:H\times H\to \mathbb{Z}$ called the {\it intersection form}. 
We denote by $\mbQ H$ the $\mbQ$-vector space with basis the set $H$;
\[
\mbQ H :=\left\{ \sum _{i=1}^n c_i[x_i] \mid n \in \mbN ,c_i \in \mbQ ,x_i \in H \right\} ,
\]
where $[-]:H\rightarrow \mbQ H$ is the embedding as basis.
Let $[-,-]:\mbQ H\times \mbQ H\rightarrow \mbQ H$ be a bilinear form defined
by $[[x],[y]]:=\langle x,y\rangle [x+y]$ for $x,y\in H$.
It is easy to see that this bilinear form is skew-symmetric and satisfies
the Jacobi identity.
We call the $\mbQ$-vector space $\mbQ H$ equipped with the Lie bracket $[-,-]$
the {\it homological Goldman Lie algebra of $\Sigma$}.
This Lie algebra was originally introduced by Goldman \cite{G} p.295--p.297.

The purpose of this paper is to study generators of $\mbQ H$.
In the previous paper \cite{T}, the third-named author determined the ideals of $\mbQ H$.
In particular, it follows that the abelianization of $\mbQ H$ is finite dimensional
if and only if the intersection form $\langle -,-\rangle$ is non-degenerate;
otherwise $\mbQ H$ is not finitely generated as a Lie algebra.
On the other hand, the abelianization of $\mathbb{Z}H$ is not finitely generated 
even in the non-degenerate case.  

Hereafter we assume that $\langle -,-\rangle$ is non-degenerate, i.e.,
$\Sigma$ is closed or has one boundary component.
Then there exists a $\mbZ$-basis $\{ A_i, B_i\} _{i=1}^g$ of $H$, called
a {\it symplectic basis} of $H$, such that
\[
\langle A_i,B_j\rangle =\delta _{ij},\quad
\langle A_i,A_j\rangle = \langle B_i,B_j\rangle = 0
\]
for all $i,j\in \{1,\ldots ,g\}$, 
where $\delta _{ij}$ is Kronecker's delta.  
Throughout this paper, we fix a symplectic basis of $H$.

We will show that if $\langle -,-\rangle$ is non-degenerate,
then the Lie algebra $\mbQ H$ is finitely generated. Moreover, we
determine the minimal number of generators of $\mbQ H$
consisting of elements of $H$.
Our main theorem is the following. 
\begin{Thm}
\label{main}
Suppose the intersection form $\langle -,-\rangle$ is non-degenerate.
There is a subset $S$ of $H$ such that $\{[s] \mid s\in S\}$
generates $\mbQ H$ as a Lie algebra and $\# S=2g+2$.
In particular, the Lie algebra $\mbQ H$ is finitely generated.
Moreover, if $S$ is a subset of $H$ and
$\{[s] \mid s\in S\}$ generates $\mbQ H$ as a Lie algebra,
we have $\#S \geq 2g+2$. 
\end{Thm}

\begin{rem}
As we see in Lemma \ref{second} and Theorem \ref{Upper},
$\mbQ H$ is generated by primitive elements of $H$ and $[0]$.
By Theorem \ref{Lower}, if $S$ is a subset of $H$ and
$\{[s] \mid s\in S\}$ generates $\mbQ H$, then $0\in S$.
Therefore, the quotient Lie algebra $\mbQ H/\mathbb{Q} [0]$
is generated by $2g+1$ primitive elements of $H$, and $2g+1$
is the minimal number of generators of $\mbQ H/\mathbb{Q} [0]$
consisting of elements of $H$. This result reminds us a result
of Humphries \cite{H} that the mapping class group of
a closed oriented surface is generated by $2g+1$ Dehn twists,
and $2g+1$ is the minimal number of generators
consisting of Dehn twists.
Although we do not see any
relationship between Humphries's result and our result,
this coincidence seems interesting.

On the other hand,
there is another Lie algebra associated to the oriented surface $\Sigma$
called the {\it Goldman Lie algebra}. This Lie algebra is denoted by
$\mathbb{Q}\hat{\pi}$, where $\hat{\pi}$ is the set of
homotopy classes of oriented loops on $\Sigma$. For more details, see Goldman \cite{G}.
The natural projection $\hat{\pi}\to H$ induces a surjective Lie
algebra homomorphisms $\mathbb{Q}\hat{\pi} \to \mbQ H$ and
$\mathbb{Q}\hat{\pi}/\mathbb{Q} 1\to \mbQ H/\mathbb{Q} [0]$.
Here $1$ is the homotopy class of a constant loop.
Simple closed curves on $\Sigma$ are analogous to primitive elements of $H$.
However, in contrast with the above property of $\mbQ H$, the quotient Lie algebra
$\mathbb{Q}\hat{\pi}/\mathbb{Q}1$ is not generated by simple closed curves.
This follows from the fact that the kernel of the Turaev cobracket \cite{Tu}
is a proper Lie subalgebra of $\mathbb{Q}\hat{\pi}/\mathbb{Q}1$, and
simple closed curves are in the kernel of the Turaev cobracket.
It is not known whether the Goldman Lie algebra $\mathbb{Q}\hat{\pi}$
for a surface which is closed or has one boundary component
is finitely generated or not.
\end{rem}

The rest of this paper is devoted to the proof of Theorem \ref{main}.

\section{Upper bound of the number of generators}

We define the $\mbQ$-linear map $ad:\mbQ H\rightarrow {\rm End}(\mbQ H)$ by $ad(X)(Y):=[X,Y]$.

\begin{Lem}\label{first}
Fix $k\in \{1,\ldots ,g\}$. 
Let $\mathfrak{g}$ be a Lie subalgebra of $\mbQ H$ such that $[\pm A_k],[\pm B_k]\in \mathfrak{g}$. 
If $X\in H$ satisfies $\langle A_k,X\rangle =0$,
$\langle B_k,X\rangle=0$ and $[X+A_k]\in \mathfrak{g}$, 
then we have $[X+aA_k+bB_k]\in \mathfrak{g}$ for all $(a,b)\in \mbZ ^2\setminus \{(0,0)\}$. 
\end{Lem}
\begin{proof}
First of all, we have $[X+aA_k]\in \mathfrak{g}$ for all $a\in \mbZ \setminus \{0\}$
since
\begin{align*}
ad([-B_k])ad([A_k])^{a-1}ad([B_k])([X+A_k])&=-a[X+a A_k] \quad {\rm if\ } a>0, {\rm \ and} \\
ad([B_k])ad([-A_k])^{-a+1}ad([-B_k])([X+A_k])&=-a[X+aA_k] \quad {\rm if\ } a<0.
\end{align*}
Similarly, we have $[X+bB_k]\in \mathfrak{g}$ for all $b\in \mbZ \setminus \{0\}$ since
\begin{align*}
ad([-A_k])ad([B_k])^b([X+A_k])&=(-1)^{b+1}b[X+bB_k] \quad {\rm if\ } b>0, {\rm \ and} \\ 
ad([-A_k]ad([-B_k])^{-b}([X+A_k])&=-b[X+bB_k] \quad {\rm if\ } b<0.
\end{align*}
Therefore we have $[X+aA_k+bB_k]\in \mathfrak{g}$ if $ab=0$ and $(a,b)\neq (0,0)$.

Suppose $ab\neq 0$. By what we have just proved, we have $[X+aA_k]\in \mathfrak{g}$. 
Then we have $[X+aA_k+bB_k]\in \mathfrak{g}$ since 
\begin{align*}
ad([B_k])^b([X+aA_k])&=(-a)^b[X+aA_k+bB_k] \quad {\rm if\ } b>0, {\rm \ and} \\
ad([-B_k])^{-b}([X+aA_k])&=a^{-b}[X+aA_k+bB_k]  \quad {\rm if\ } b<0.
\end{align*}
This completes the proof.
\end{proof}

\begin{Lem}\label{second}
The set
\[
\{ [\pm A_i](1\leq i\leq g),[\pm B_i](1\leq i\leq g),[A_i+A_j](1\leq i<j\leq g),[0] \}
\]
generates $\mbQ H$ as a Lie algebra. In particular, $\mbQ H$ is finitely generated. 
\end{Lem}
\begin{proof}
Let $\mathfrak{g}$ be the Lie subalgebra generated by the above set. 

\noindent {\bf Claim 1.}
{\it For any integer $n>0$ and indices $i_1,\ldots,i_n$ with
$1\leq i_1<\cdots <i_n\leq g$, we have $[A_{i_1}+\cdots +A_{i_n}]\in \mathfrak{g}$. }

We prove Claim 1 by induction on $n$. 
If $n=1$ or $n=2$, the claim follows from the assumption of the lemma. 
Suppose $n>2$ and let $i_1,\ldots ,i_n$ be indices with $1\leq i_1 <\cdots <i_n \leq g$.
By the inductive assumption,
we have $[A_{i_1}+\cdots +A_{i_{n-1}}]\in \mathfrak{g}$. 
Since
\begin{eqnarray*}
ad([-B_{i_1}])ad([-A_{i_1}])ad([A_{i_1}+A_{i_n}])ad([B_{i_1}])([A_{i_1}+\cdots +A_{i_{n-1}}]) \\
=[A_{i_1}+\cdots +A_{i_n}],
\end{eqnarray*}
we have $[A_{i_1}+\cdots +A_{i_n}]\in \mathfrak{g}$.
This proves Claim 1.

Let $x\in H$ be an arbitrary element.
If $x=0$, we have $[x]\in \mathfrak{g}$ by the assumption of the lemma. 
Suppose $x\neq 0$. Since $\{A_i,B_i\}_{i=1}^g$ is a $\mathbb{Z}$-basis of $H$, we can write
$x=a_1A_{i_1}+b_1B_{i_1}+\cdots +a_nA_{i_n}+b_nB_{i_n}$ with
$1\leq i_1<\cdots <i_n\leq g$ and $(a_1,b_1),\ldots ,(a_n,b_n)\in \mbZ ^2\setminus \{(0,0)\}$.

\noindent {\bf Claim 2.}
{\it We have $[a_1A_{i_1}+b_1B_{i_1}+\cdots
+a_mA_{i_m}+b_mB_{i_m}+A_{i_{m+1}}+\cdots +A_{i_n}]\in \mathfrak{g}$ for all $m=1,\ldots ,n$.}

We prove Claim 2 by induction on $m$. 
Suppose $m=1$. We have $[A_{i_2}+\cdots +A_{i_n}]\in \mathfrak{g}$ by Claim 1.
Applying Lemma \ref{first} to $k=i_1$ and $X=A_{i_2}+\cdots +A_{i_n}$,
we have $[a_1A_{i_1}+b_1B_{i_1}+A_{i_2}+\cdots +A_{i_n}]\in \mathfrak{g}$.
This proves the case $m=1$.

Suppose $m>1$. By the inductive assumption,
we have $[a_1A_{i_1}+b_1B_{i_1}+\cdots +a_{m-1}A_{i_{m-1}}+b_{m-1}B_{i_{m-1}}+A_{i_{m+1}}+\cdots +A_{i_n}]\in \mathfrak{g}$. Applying Lemma \ref{first} to $k=i_m$ and
$X=a_1A_{i_1}+b_1B_{i_1}+\cdots +a_{m-1}A_{i_{m-1}}+b_{m-1}B_{i_{m-1}}+A_{i_{m+1}}+\cdots +A_{i_n}$,
we obtain $[a_1A_{i_1}+b_1B_{i_1}+\cdots +a_mA_{i_m}+b_mB_{i_m}+A_{i_{m+1}}+\cdots +A_{i_n}]\in \mathfrak{g}$. This proves Claim 2.

Applying Claim 2 to $m=n$, we have $[x]\in \mathfrak{g}$.
Since $\{ [x] \mid x\in H\}$ is a $\mathbb{Q}$-basis of $\mbQ H$, we obtain
$\mathfrak{g}=\mbQ H$. This completes the proof.
\end{proof}

Now we give generators of $\mbQ H$ consisting of $2g+2$ elements of $H$.
\begin{Thm}\label{Upper}
The set 
\[
\{ [A_i](1\leq i\leq g),[B_i](1\leq i\leq g),[-A_1-\cdots -A_g-B_1-\cdots -B_g],[0]\}
\]
generates $\mbQ H$ as a Lie algebra. 
\end{Thm}
\begin{proof}
Let $\mathfrak{g}$ be the Lie subalgebra generated by the above set. 
Set $X:=[-A_1-\cdots -A_g-B_1-\cdots -B_g]$. 

First of all, we have $[A_1+\cdots +A_g],[B_1+\cdots +B_g]\in \mathfrak{g}$ since
\begin{align*}
ad([B_1])\cdots ad([B_g])ad([A_1])^2\cdots ad([A_g])^2(X)&=(-1)^g[A_1+\cdots +A_g], \mbox{ and} \\
ad([A_1])\cdots ad([A_g])ad([B_1])^2\cdots ad([B_g])^2(X)&=[B_1+\cdots +B_g].
\end{align*}
Next, we have $[-A_i],[-B_i]\in \mathfrak{g}$ for all $i=1,\ldots ,g$ since
\begin{eqnarray*}
&ad([B_1+\cdots +B_g])ad([A_1])\cdots ad([A_{i-1}])ad([A_{i+1}])\cdots ad([A_g])(X)\\
&=(-1)^{g-1}[-A_i], \mbox{ and} \\
&ad([A_1+\cdots +A_g])ad([B_1])\cdots ad([B_{i-1}])ad([B_{i+1}])\cdots ad([B_g])(X)\\
&=-[-B_i]. 
\end{eqnarray*}
Finally, we have $[A_i+A_j]\in \mathfrak{g}$ for $1\leq i<j \leq g$ since
\[
ad([B_1+\cdots +B_g])ad([A_1+\cdots +A_g])ad([A_i])ad([A_j])(X)=2g[A_i+A_j].
\]
Now the assertion follows from Lemma \ref{second}. 
\end{proof}

\section{Lower bound of the number of generators}
\begin{Lem}\label{monoid}
Let $S$ be a subset of $H$, $\mathfrak{g}$ the Lie subalgebra generated by $\{ [s] \mid s\in S\}$,
and $M$ the submonoid in $H$ generated by $S$, i.e., 
\[
M=\{ s_1+\cdots +s_n \in H \mid n\in \mbN ,s_i\in S\}.
\]
Then, we have $\mathfrak{g}\subset \mbQ M$. 
\end{Lem}
\begin{proof}
The set $\{ [s_1,[s_2,[\cdots ,[s_{n-1},s_n]\cdots ]]] \mid n\in \mbN ,s_i\in S\}$ generates $\mathfrak{g}$ as a $\mbQ$-vector space. Since
\[
[s_1,[s_2,[\cdots ,[s_{n-1},s_n]\cdots ]]]=
( \prod _{i=1}^{n-1}\langle s_i,\sum _{j=i+1}^ns_j\rangle )[s_1+\cdots +s_n]\in \mbQ M, 
\]
we obtain $\mathfrak{g} \subset \mbQ M$.
\end{proof}

\begin{Thm}\label{Lower}
Let $S$ be a subset of $H$. 
If $\{ [s] \mid s\in S\}$ generates $\mbQ H$ as a Lie algebra,
we have $0\in S$ and $\# S \geq 2g+2$. 
\end{Thm}
\begin{proof}
Suppose $\mbQ H$ is generated by $\{ [s] \mid s\in S\}$ as a Lie algebra.
Now, we have $[\mbQ H,\mbQ H]\subset \mbQ (H\setminus \{0\})$ since $\langle x,y\rangle =0$ for
$x,y\in H$ with $x+y=0$. 
This implies that $0\in S$. 

Let $M$ be the submonoid generated by $S$. 
By Lemma \ref{monoid} we have $\mbQ H \subset \mbQ M$, thus $H\subset M$. 
Since $H\supset M$, we obtain $H=M$. 
In other words, the set $S$ generates $H$ as a monoid. 
In particular, $S\setminus \{0\}$ generates $H$ as a $\mbZ$-module. 
Since $H$ is a free $\mbZ$-module of rank $2g$, we have $\# (S\setminus \{0\})\geq 2g$. 

If $\# (S\setminus \{0\})= 2g$, $S\setminus \{0\}$ is a $\mathbb{Z}$-basis of $H$. 
Then $M=\{ \sum _{s\in S\setminus \{0\}} a_ss \mid a_s\geq 0\} \neq H$, which contradicts $M=H$. 
Hence $\# (S\setminus \{0\})> 2g$, therefore we have $\# S\geq 2g+2$. 
\end{proof}

By Theorems \ref{Upper} and \ref{Lower}, we obtain Theorem \ref{main}.

\begin{rem}
Let $R$ be a commutative ring with unit. Using $R$ instead of $\mbQ$,
we can similarly define the homological Goldman Lie algebra $RH$.
If $R$ includes $\mbQ$, the same result as Theorem \ref{main} holds for $RH$.
\end{rem}

In Theorem \ref{Lower}, we assumed $S\subset H$. 
This condition is essential. 
If we consider generators which are not necessarily
elements of $H$, we can find generators of $\mbQ H$ whose
number is less than $2g+2$.
\begin{Prop}
The Lie algebra $\mbQ H$ is generated by $g+2$ elements
$[A_1],\ldots,[A_g]$,
$[-A_1-\cdots -A_g-B_1-\cdots -B_g]+[B_1]+\cdots +[B_g]+[0]$,
and $[B_1+\cdots +B_g]$.
\end{Prop}
\begin{proof}
Let $\mathfrak{g}$ be the Lie subalgebra generated by $[A_1],\ldots,[A_g]$,
$X:=[-A_1-\cdots -A_g-B_1-\cdots -B_g]+[B_1]+\cdots +[B_g]+[0]$, and $Y:=[B_1+\cdots +B_g]$. 

Suppose $g=1$. Then $X=[-A_1-B_1]+[B_1]+[0]$ and $Y=[B_1]$.
We have $[-A_1]\in \mathfrak{g}$ since $[X,Y]=-[-A_1]$, and
we have $[-B_1]\in \mathfrak{g}$ since 
$[X-Y,[A_1]]=[-B_1]$. 
Then we have $[-A_1-B_1]\in \mathfrak{g}$ since $[[-A_1],[-B_1]]=[-A_1-B_1]$.
Since $[0]=X-[-A_1-B_1]-Y$ we have $[0]\in \mathfrak{g}$. 
Therefore we obtain $\{ [A_1],[B_1],[-A_1-B_1],[0]\} \subset \mathfrak{g}$,
which implies $\mathfrak{g}=\mbQ H$ by Theorem \ref{Upper}. 

Suppose $g\geq 2$. 
Then $(-1)^g[-B_1\cdots -B_g] = ad([A_1])\cdots ad([A_g])(X) \in \mathfrak{g}$. 
On the other hand, we have $-g[-A_1\cdots -A_g] = [X, Y] \in \mathfrak{g}$.
Set $Z=[-A_1\cdots -A_g-B_1\dots -B_g]$. Then $Z\in \mathfrak{g}$ since
\[
[[-A_1-\cdots -A_g],[-B_1-\cdots -B_g]]=gZ. 
\]
We have $[-A_i]\in \mathfrak{g}$ for $i=1,\ldots ,g$ since 
\[
ad(Y)ad([A_1])\cdots ad([A_{i-1}])ad([A_{i+1}])\cdots ad([A_g])(Z)=(-1)^{g-1}[-A_i], 
\]
and we have $[B_i]\in \mathfrak{g}$ for $i=1,\ldots ,g$ since 
\[
ad([-A_i])ad([A_i])(X)+Z=-[B_i] . 
\]
Finally, we have $[0]\in \mathfrak{g}$ since $[0]=X-Z-[B_1]-\cdots -[B_g]$. 
Hence, we have $\mathfrak{g}=\mbQ H$ by Theorem \ref{Upper}.
\end{proof}

\noindent \textsc{Nariya Kawazumi\\
Department of Mathematical Sciences,\\
University of Tokyo,\\
3-8-1 Komaba Meguro-ku Tokyo 153-8914 JAPAN}\\
\noindent \texttt{E-mail address: kawazumi@ms.u-tokyo.ac.jp}

\vspace{0.5cm}

\noindent \textsc{Yusuke Kuno\\
Department of Mathematics,\\
Tsuda College,\\
2-1-1, Tsuda-Machi, Kodaira-shi, Tokyo 187-8577 JAPAN}\\
\noindent \texttt{E-mail address: kunotti@tsuda.ac.jp}

\vspace{0.5cm}

\noindent \textsc{Kazuki Toda\\
Graduate School of Mathematical Sciences, \\
University of Tokyo, \\
3-8-1 Komaba Meguro-ku, Tokyo 153-8914, JAPAN}\\
\noindent \texttt{E-mail address: ktoda@ms.u-tokyo.ac.jp}

\end{document}